\documentclass[reqno, 12pt]{amsart}
\usepackage{amsmath, amsthm, amsfonts, amssymb}
\usepackage{color,xcolor}

\newtheorem{thm}{Theorem}[section]
 
 \newtheorem{lem}[thm]{Lemma}

 \newtheorem{prop}[thm]{Proposition}

 \theoremstyle{definition}
 \newtheorem{definition}[thm]{Definition}
\newtheorem{rem}[thm]{Remark}

 \theoremstyle{remark}

 \numberwithin{equation}{section}

\DeclareMathOperator{\lcm}{lcm}

 \allowdisplaybreaks
 
\begin{document}

\title{Weighted Yosida Mappings of Several Complex Variables}
\author[N. Bharti]{Nikhil Bharti}
\address{
\begin{tabular}{lll}
&Nikhil Bharti\\
&Department of Mathematics\\
&University of Jammu\\
&Jammu-180006\\ 
&India\\
\end{tabular}}
\email{nikhilbharti94@gmail.com}

\author[N. V. Thin]{Nguyen Van Thin$^{*}$}
\address{
\begin{tabular}{lll}
&Nguyen Van Thin\\
&Department of Mathematics\\
&Thai Nguyen University of Education\\
&Luong Ngoc Quyen street\\
&Thai Nguyen City, Thai Nguyen\\ 
&Vietnam\\
\end{tabular}}
\email{thinmath@gmail.com}
\thanks{$^*$ Corresponding author: Nguyen Van Thin}

\begin{abstract}
Let $M$ be a complete complex Hermitian manifold with metric $E_{M}$ and let $\varphi: [0,\infty)\rightarrow (0,\infty)$ be positive function such that $$\gamma_r=\sup\limits_{r\leq a<b}\left|(\varphi(a)-\varphi(b))/(a-b)\right|\leq C,~r\in (0,\infty),$$ for some $C\in (0,1],$ and $\lim_{r\rightarrow\infty}\gamma_r=0.$ A holomorphic mapping $f:\mathbb{C}^{m}\rightarrow M$ is said to be a weighted Yosida mapping if for any $z,~\xi\in\mathbb{C}^{m}$ with $\|\xi\|=1,$ the quantity $\varphi(\|z\|)E_{M}(f(z), df(z)(\xi))$ remains bounded above, where $df(z)$ is the map from $T_z(\mathbb{C}^{m})$ to $T_{f(z)}(M)$ induced by $f.$ We present several criteria of holomorphic mappings belonging to the class of all weighted Yosida mappings.
\end{abstract}

\renewcommand{\thefootnote}{\fnsymbol{footnote}}

\maketitle

\section{\bf Introduction}
The main objective of this paper is to extend the concept of weighted Yosida functions to holomorphic mappings of several complex variables. A function $f$ which is meromorphic in $\mathbb{C},$ belongs to the class $\mathcal{Y}_{\varphi}(\mathbb{C})$ of weighted Yosida functions if 
\begin{equation}\label{eq:I-1}
\limsup\limits_{|z|\to\infty}\varphi(|z|)f^{\#}(z)<\infty,
\end{equation}
 where $$f^{\#}(z)=\frac{|f'(z)|}{1+|f(z)|^2}$$ is the spherical derivative of $f,$ and $\varphi: [0,\infty)\rightarrow (0,\infty)$ is a  positive function such that $$\gamma_r:=\sup\limits_{r\leq a<b}\left|\frac{\varphi(a)-\varphi(b)}{a-b}\right|\leq C,~r\in (0,\infty),$$ for some $C\in (0,1],$ and $\lim\limits_{r\rightarrow\infty}\gamma_r=0.$ 

\medskip

If $\varphi\equiv 1,$ then the class $\mathcal{Y}_{1}(\mathbb{C})$ is the class of Yosida functions (see \cite{aulaskari, yosida}). If, in \eqref{eq:I-1}, $\varphi(z)= z^{2-p},$ for $~1\leq p<\infty,$ then $f\in\mathcal{Y}_{p}(\mathbb{C})$ and we say that $f$ is $p$-Yosida function (see \cite{makhmutov}). The study of Yosida functions plays an important role in the theory of meromorphic functions. For instance, it has been established that doubly periodic functions are Yosida functions. Moreover, Julia exceptional functions are just transcendental functions belonging to the class $\mathcal{Y}_{1}(\mathbb{C})$ (see \cite{aulaskari, makhmutov}). Recently, Yosida mappings of several complex variables has gained a lot of interest. An extension of $p$-Yosida functions to holomorphic functions of several variables obtained by Yang \cite{yang}. However, to the best of our knowledge, the study of weighted Yosida functions in several complex variables is still pending. In this paper, our aim is to fill this gap in the literature and to open new possibilities for future study. More precisely, we present several criteria for a holomorphic mapping to belong to the class of all weighted Yosida mappings. As applications, we obtain variants of Lappan's five point theorem for weighted Yosida mappings in several complex variables. Before we state our results, let us recall some standard terminologies and for more deeper insights, we refer the reader to \cite{huybrechts, szekelyhidi}.    
 
Let $M$ be a complete complex Hermitian manifold of dimension $k$ with a Hermitian metric $$ds_{M}^{2}=\sum\limits_{i,j=1}^{k}h_{ij}(p)dz_i\bar{dz_j},$$ where $z=(z_1,~\ldots,~z_k)$ are local coordinates in a neighbourhood of a point $p\in M.$ A differential form of type $(1,1)$ is associated to $ds_{M}^{2}$ as $$\omega_{M}=\sqrt{-1}\sum\limits_{i,j=1}^{k}h_{ij}(p)dz_i\wedge\bar{dz_j}.$$

We denote by $E_{M}(p,\xi),$ the metric for $M$ at $p$ in the direction of the vector $\xi\in T_{p}(M),$ where $T_{p}(M)$ is the complexified tangent space to $M$ at $p.$ Then $$E_{M}(p,\xi)=(ds_{M}^{2}(\xi, \bar{\xi}))^{1/2}.$$ Thus the distance between two points $p,~q\in M$ is given by $$d_{M}(p,q):=\inf\limits_{\gamma}\int\limits_{0}^{1}E_{M}(\gamma(t), \gamma'(t))~dt,$$ where the infimum is taken over all parametric curves $\gamma: [0,1]\rightarrow M$ with $\gamma(0)=p$ and $\gamma(1)=q.$ 

Let $M=P^{n}(\mathbb{C}),$ the $n$-dimensional complex projective space, that is, the space $\left(\mathbb{C}^{n+1}\setminus\{0\}\right)/\sim,$ where $$(x_0,\ldots, x_n)\sim (y_0,\ldots, y_n) \iff (x_0,\ldots, x_n)=\lambda(y_0,\ldots, y_n),$$ for some $\lambda\in\mathbb{C}\setminus\{0\}.$ Then a point in $P^{n}(\mathbb{C}$) is an equivalence class of some $(x_0,\ldots, x_n)\in\mathbb{C}^{n+1}\setminus\{0\},$ which is written as $[x_0,\ldots, x_n]$ and is called the homogeneous coordinate of the point. For a system of homogeneous coordinates $x=[x_0:\cdots:x_n]$ in $P^{n}(\mathbb{C}),$ the metric $ds_{P^{n}(\mathbb{C})}^{2}$ can be expressed as $$ds_{P^{n}(\mathbb{C})}^{2}=\frac{\langle dx,dx \rangle\langle x,x \rangle-|\langle x,dx \rangle|^2}{\langle x,x \rangle^2},$$ where $\langle \textbf{.},\textbf{.} \rangle$ denotes the standard Hermitian product in $\mathbb{C}^{n}.$ This is known as the {\it Fubini-Study metric} in homogeneous coordinates on $P^{n}(\mathbb{C}).$ Note that for $n=1,$ $P^{1}(\mathbb{C})$ is identified with the extended complex plane $\mathbb{C}\cup\{\infty\}.$  

\section{\bf Weighted Yosida mappings}

We now introduce weighted Yosida mappings of several complex variables. For convenience, we shall denote by $\mathcal{H}(X, Y),$ the class of all holomorphic mappings from the complex space $X$ to the complex space $Y.$ 

\begin{definition}
 Let $f\in\mathcal{H}(\mathbb{C}^{m}, M).$ Then $f$ belongs to the class $\mathcal{Y}_{\varphi}\left(\mathbb{C}^{m}, M\right)$ of weighted Yosida mappings if $$\limsup\limits_{\|z\|\rightarrow\infty,~ \|\xi\|=1}\varphi(\|z\|)E_{M}(f(z), df(z)(\xi))<\infty,$$ where $df(z)$ is the map from $T_z(\mathbb{C}^{m})$ to $T_{f(z)}(M)$ induced by $f.$
\end{definition}

In the following, we establish a necessary and sufficient condition for holomorphic mappings belonging to the class $\mathcal{Y}_{\varphi}\left(\mathbb{C}^{m}, M\right).$ Here and after, it is assumed that $M$ is compact. Recall that for a domain $D\subseteq\mathbb{C}^{m},$ a family $\mathcal{F}\subset\mathcal{H}(D, M)$ is said to be normal if every sequence of functions in $\mathcal{F}$ has a subsequence which is relatively compact in $\mathcal{H}(D, M)$ (see \cite{aladro, wu, thai}).

\begin{prop}\label{prop:1}
Let $f\in\mathcal{H}(\mathbb{C}^{m}, M).$ Then $f\in\mathcal{Y}_{\varphi}(\mathbb{C}^{m}, M)$ if and only if for all sequences $\left\{z_j\right\},~\left\{\xi_j\right\}\subset\mathbb{C}^{m}$ with $\|z_j\|\rightarrow\infty$ and $\|\xi_j\|=1,$ the family $$\mathcal{F}=\left\{F_j(\zeta):=f(z_j+\varphi(\|z_j\|)\xi_j\zeta),~j\in\mathbb{N}\right\}$$ is normal on $\mathbb{C}.$
\end{prop}

To prove Proposition \ref{prop:1}, we need the following result of Thai, Trang and Huong \cite{thai}:

\begin{lem}\label{lem:1}
Let $D\subset\mathbb{C}^{m}$ be a domain and let $M$ be a complex space. Let $\mathcal{F}\subset\mathcal{H}(\mathbb{C}^{m}, M)$ and suppose that $\mathcal{F}$ is normal on $D.$ Then for each norm $E_M$ on $T(M)$ and for each compact subset $K$ of $D,$ there is a constant $C_K>0$ such that 
\begin{equation}\label{eq:1}
E_M(f(z), df(z)(\xi))\leq C_K\|\xi\|
\end{equation}
for every $z\in K,~\xi\in\mathbb{C}^{m}\setminus\left\{0\right\}$ and $f\in\mathcal{F}.$ Conversely, if $M$ is a complete Hermitian complex space and if the family $\mathcal{F}$ is not completely divergent such that \eqref{eq:1} holds, the $\mathcal{F}$ is normal on $D.$
\end{lem}

\begin{proof}[\bf Proof of Proposition \ref{prop:1}]
Suppose that $f\in\mathcal{Y}_{\varphi}(\mathbb{C}^{m}, M),~\left\{z_j\right\},~\left\{\xi_j\right\}\subset\mathbb{C}^{m},$ with $\|z_j\|\rightarrow\infty$ and $\|\xi_j\|=1.$ Then by definition of weighted Yosida mappings, there exists $K_0>0$ such that for $z,~\xi\in\mathbb{C}^{m}$ with $\|\xi\|=1,$ we have $$\varphi(\|z\|)E_{M}(f(z), df(z)(\xi))\leq K_0.$$ Let $R>0$ and $w_j:=z_j+\varphi(\|z_j\|)\xi_j\zeta,~|\zeta|<R.$
Then 
\begin{align}\label{eq:2}
E_{M}(F_j(\zeta), dF_j(\zeta)(1)) &=\varphi(\|z_j\|)E_{M}(f(w_j), df(w_j)(\xi_j))\nonumber\\
&=\frac{\varphi(\|z_j\|)}{\varphi(\|w_j\|)}\cdot\varphi(\|w_j\|) E_{M}(f(w_j), df(w_j)(\xi_j))\nonumber\\
&\leq K_0\frac{\varphi(\|z_j\|)}{\varphi(\|w_j\|)}.
\end{align}
Since $\|w_j-z_j\|<\varphi(\|z_j\|)R,$ we have
$$\left|\frac{\varphi(\|w_j\|)}{\varphi(\|z_j\|)}-1\right|=\frac{\left|\varphi(\|w_j\|)-\varphi(\|z_j\|)\right|}{\varphi(\|z_j\|)}\leq \gamma_r\frac{\left|\|w_j\|-\|z_j\|\right|}{\varphi(\|z_j\|)}<\gamma_r\frac{\varphi(\|z_j\|)}{\varphi(\|z_j\|)}R=\gamma_rR.$$ Since $\gamma_r\rightarrow 0$ as $r\rightarrow\infty,$ so $\varphi(\|w_j\|)/\varphi(\|z_j\|)\rightarrow 1$ as $\|z_j\|\rightarrow\infty.$ Thus from \eqref{eq:2} and Lemma \ref{lem:1}, we conclude that the family $\mathcal{F}=\left\{F_j(z):=f(z_j+\varphi(\|z_j\|)\xi_j\zeta),~j\in\mathbb{N}\right\}$ is normal on $\mathbb{C}.$\\
Conversely, suppose, on the contrary, that $f\notin \mathcal Y_{\varphi}(\mathbb{C}^{m}, M).$ Then there exist $\left\{z_j\right\},~\left\{\xi_j\right\}\subset\mathbb{C}^{m},$ with $\|z_j\|\rightarrow\infty$ and $\|\xi_j\|=1$ such that $$\lim\limits_{j\rightarrow\infty}\varphi(\|z_j\|)E_{M}(f(z_j), df(z_j)(\xi))=\infty.$$ Now, consider the family $\mathcal{F}=\left\{F_j(z):=f(z_j+\varphi(\|z_j\|)\xi_j\zeta),~j\in\mathbb{N}\right\}.$ Then $$E_{M}(F_j(0), dF_j(0)(1))=\varphi(\|z_j\|)E_{M}(f(z_j), df(z_j)(\xi))\rightarrow\infty\mbox{ as } j\rightarrow\infty.$$ Clearly, by Lemma \ref{lem:1}, $\mathcal{F}$ is not normal on $\mathbb{C}.$
\end{proof}

\begin{rem}\label{rem:1} Let $f\in\mathcal{Y}_{\varphi}(\mathbb{C}^{m}, M).$ Then by Proposition \ref{prop:1}, for all sequences $\left\{z_j\right\},~\left\{\xi_j\right\}\subset\mathbb{C}^{m}$ with $\|z_j\|\rightarrow\infty$ and $\|\xi_j\|=1,$ the family $\mathcal{F}=\left\{F_j(\zeta):=f(z_j+\varphi(\|z_j\|)\xi_j\zeta),~j\in\mathbb{N}\right\}$ is normal on $\mathbb{C}.$ If, in addition, no subsequence of $\mathcal{F}$ has a constant limit, then we say that $f$ is a {\it weighted Yosida mapping of first category} and write it as $f\in\mathcal{Y}_{\varphi}^{0}(\mathbb{C}^{m}, M).$
\end{rem}

With Remark \ref{rem:1} in hand, it is natural to ask when a weighted Yosida mapping belongs to the class $\mathcal{Y}_{\varphi}^{0}(\mathbb{C}^{m}, M).$ The following result provides an answer to this problem:

\begin{thm}\label{thm:1.1} Let $f\in\mathcal{Y}_{\varphi}(\mathbb{C}^{m}, M),~D_{\varphi}(a, \epsilon):=\{z\in\mathbb{C}^{m}: \|z-a\|<\epsilon\varphi(\|a\|)$ and $D(a, \epsilon):=\{z\in\mathbb{C}^{m}: \|z-a\|<\epsilon.$ Then $f\in\mathcal{Y}_{\varphi}^{0}(\mathbb{C}^{m}, M)$ if and only if the any one of following conditions hold:
\begin{itemize}
\item[(i)] For any $\epsilon>0,$ $$\liminf\limits_{\|a\|\to\infty}\max\limits_{z\in D_{\varphi}(a, \epsilon),~\|\xi\|=1}\varphi(\|z\|)E_{M}(f(z), df(z)(\xi))>0.$$
\item[(ii)] For any $\epsilon>0,$ $$\liminf\limits_{\|a\|\to\infty}\int_{z\in D_{\varphi}(a, \epsilon)}f^*\omega_M(z)>0,$$ where $f^*\omega_M$ is the pullback of the metric $\omega_M$ under $f.$
\end{itemize}
\end{thm}

\begin{proof}
First note that for any $\epsilon>0,~f\in\mathcal{H}(\mathbb{C}^{m}, M)$ and sequences $\left\{z_j\right\},~\left\{\xi_j\right\}\subset\mathbb{C}^{m}$ with $\|\xi_j\|=1,$ $$\max\limits_{z\in D_{\varphi}(z_j, \epsilon)}\varphi(\|z\|)E_{M}(f(z), df(z)(\xi_j))=\max\limits_{\zeta\in D(0,\epsilon)}\frac{\varphi(\|z_j+\varphi(\|z_j\|)\xi_j\zeta\|)}{\varphi(\|z_j\|)}E_{M}(F_j(\zeta), dF_j(\zeta)(\xi_j))$$  and $$\int_{z\in D_{\varphi}(z_j, \epsilon)}f^*\omega_M(z)=\int_{\zeta\in D(0,\epsilon)}F^{*}_{j}\omega_M(\zeta),$$ where $F_j(\zeta)=f(z_j+\varphi(\|z_j\|\xi_j\zeta)).$ Now, suppose that $f\in\mathcal{Y}_{\varphi}^{0}(\mathbb{C}^{m}, M).$ Then by Proposition \ref{prop:1}, we find that the family $\left\{F_j(\zeta)\right\}$ is normal on $\mathbb{C}.$ Assume that $\left\{F_j(\zeta)\right\}$ converges locally uniformly on $\mathbb{C}$ to a non-constant mapping $h\in\mathcal{H}(\mathbb{C}, M)$ and $\xi_j\rightarrow \xi^*$ as $j\rightarrow\infty,$ where $\|\xi^*\|=1.$ Then 
\begin{align*}
& \lim\limits_{j\to\infty}\max\limits_{z\in D_{\varphi}(z_j, \epsilon)}\varphi(\|z\|)E_{M}(f(z), df(z)(\xi_j))\\
 &= \lim\limits_{j\to\infty}\max\limits_{\zeta\in D(0,\epsilon)}\frac{\varphi(\|z_j+\varphi(\|z_j\|)\xi_j\zeta\|)}{\varphi(\|z_j\|)}E_{M}(F_j(\zeta), dF_j(\zeta)(\xi_j))\\
&= \lim\limits_{j\to\infty}\max\limits_{\zeta\in D(0,\epsilon)}E_{M}(F_j(\zeta), dF_j(\zeta)(\xi_j))\\
&= \max\limits_{\zeta\in D(0,\epsilon)}E_{M}(h(\zeta), dh(\zeta)(\xi^*))>0.
\end{align*}
Also, $$\lim\limits_{j\to\infty}\int_{z\in D_{\varphi}(z_j, \epsilon)}f^{*}\omega_M(z)=\lim\limits_{j\to\infty}\int_{\zeta\in D(0, \epsilon)}F_j^{*}\omega_M(\zeta)=\int_{\zeta\in D(0, \epsilon)}h^{*}\omega_M(\zeta)>0.$$ This establishes the necessity of (i) and (ii).\\
Conversely, suppose that (i) or (ii) holds and $f\not\in\mathcal{Y}_{\varphi}^{0}(\mathbb{C}^{m}, M).$ Then there exist sequences $\{z_j\},~\{\xi_j\}\subset\mathbb{C}^{m}$ such that the family $\left\{F_j(z)=f(z_j+\varphi(\|z_j\|\xi_j\zeta)\right\}$ converges uniformly in a neighbourhood of the origin to a constant. Then 
\begin{align*}
& \lim\limits_{j\to\infty}\max\limits_{z\in D_{\varphi}(z_j, \epsilon)}\varphi(\|z\|)E_{M}(f(z), df(z)(\xi_j))\\
 &= \lim\limits_{j\to\infty}\max\limits_{\zeta\in D(0,\epsilon)}\frac{\varphi(\|z_j+\varphi(\|z_j\|)\xi_j\zeta\|)}{\varphi(\|z_j\|)}E_{M}(F_j(\zeta), dF_j(\zeta)(\xi_j))=0
\end{align*}
and $$\lim\limits_{j\to\infty}\int_{z\in D_{\varphi}(z_j, \epsilon)}f^{*}\omega_M(z)=\lim\limits_{j\to\infty}\int_{\zeta\in D(0, \epsilon)}F_j^{*}\omega_M(\zeta)=0.$$ This contradict both (i) and (ii), as desired.
\end{proof}

\begin{rem}
Theorem \ref{thm:1.1} for meromorphic functions and $p$-Yosida mappings can be found in \cite{makhmutov} and \cite{yang} respectively.
\end{rem}

\medskip

It is evident from Proposition \ref{prop:1} that a holomorphic mapping $f\notin\mathcal{Y}_{\varphi}(\mathbb{C}^{m}, M)$ if and only if there exist sequences $\left\{z_j\right\},~\left\{\xi_j\right\}\subset\mathbb{C}^{m}$ with $\|z_j\|\rightarrow\infty$ and $\|\xi_j\|=1,$ such that the family $$\mathcal{F}=\left\{F_j(\zeta):=f(z_j+\varphi(\|z_j\|)\xi_j\zeta),~j\in\mathbb{N}\right\}$$ is not normal on $\mathbb{C}.$ Thus, to check whether a holomorphic mapping belongs to the class $\mathcal{Y}_{\varphi}(\mathbb{C}^{m}, M),$ it is often convenient to suppose otherwise and check the non-normality of the family $\{F_j(\zeta)\}$ on $\mathbb{C}.$ This leads to the well known Zalcman's rescaling lemma \cite{zalcman} which is an essential tool to check the non-normality of a family of meromerophic functions on $\mathbb{C}.$ A great deal of focus has been given to find higher dimensional analogues of Zalcman's lemma (see \cite{aladro, chang, dovbush, thai}). Recently, Yang \cite{yang} established the Zalcman type rescaling for $p$-Yosida mappings of several complex variables. In the following, we obtain the Zalcman type rescaling result for weighted Yosida mappings:

\begin{thm}\label{thm:1} A holomorphic mapping $f\notin \mathcal{Y}_{\varphi}(\mathbb{C}^{m}, M)$ if and only if there exist sequences $\left\{z_j\right\},~\left\{\xi_j\right\}\subset\mathbb{C}^{m},~\left\{\rho_j\right\}\subset(0,1)$ with $\|z_j\|\rightarrow\infty,~\rho_j\rightarrow 0$ and $\|\xi_j\|=1$ such that the sequence $$g_j(\zeta):=f(z_j+\rho_j\varphi(\|z_j\|)\xi_j\zeta)$$ converges uniformly on compact subsets of $\mathbb{C}$ to a non-constant holomorphic mapping $g:\mathbb{C}\rightarrow M.$
\end{thm}

\begin{proof}
Suppose that $f\notin \mathcal{Y}_{\varphi}(\mathbb{C}^{m}, M).$ Then there exist sequences $\{\tilde{z_j}\},~\{\tilde{\xi_j}\}\subset\mathbb{C}^{m}$ with $\|\tilde{z_j}\|\rightarrow\infty$ and $\|\tilde{\xi_j}\|=1$ such that $$A_j:=\varphi(\|\tilde{z_j}\|)E_{M}(f(\tilde{z_j}), df(\tilde{z_j})(\tilde{\xi_j}))\rightarrow\infty \mbox{ as } j\rightarrow\infty.$$ Let $t_j=A_{j}^{-1/2},~D_{\varphi}(\tilde{z_j}, t_j)=\left\{z\in\mathbb{C}^{m}: \|z-\tilde{z_j}\|<t_j\varphi(\|\tilde{z_j}\|)\right\}$ and define $$R_j:=\sup\limits_{z\in D_{\varphi}(\tilde{z_j}, t_j),~\|\xi\|=1}\left(t_j\varphi(\|\tilde{z_j}\|)-\|z-\tilde{z_j}\|\right)E_{M}(f(z), df(z)(\xi)).$$ 
Since each function $\left(t_j\varphi(\|\tilde{z_j}\|)-\|z-\tilde{z_j}\|\right)E_{M}(f(z), df(z)(\xi))$ is continuous on $D_{\varphi}(\tilde{z_j}, t_j),$ we can find $z_j\in D_{\varphi}(\tilde{z_j}, t_j)$ and $\xi_j\in\mathbb{C}^{m}$ with $\|\xi_j\|=1$ such that $$\left(t_j\varphi(\|\tilde{z_j}\|)-\|z_j-\tilde{z_j}\|\right)E_{M}(f(z_j), df(z_j)(\xi_j))=R_j.$$ Clearly, $R_j\geq t_j\varphi(\|\tilde{z_j}\|)E_{M}(f(\tilde{z_j}), df(\tilde{z_j})(\tilde{\xi}))=t_{j}A_j=A_{j}^{1/2}\rightarrow\infty\mbox{ as } j\rightarrow\infty.$

Define $$\rho_j:=\frac{1}{\varphi(\|z_j\|)E_{M}(f(z_j), df(z_j)(\xi_j))}.$$ Then $\rho_j\rightarrow 0 \mbox{ as } j\rightarrow\infty.$\\ Let $$g_j(\zeta):=f(z_j+\rho_j\varphi(\|z_j\|)\xi_j\zeta),~|\zeta|<\frac{\left(t_j\varphi(\|\tilde{z_j}\|)-\|z_j-\tilde{z_j}\|\right)}{\rho_j\varphi(\|z_j\|)}=R_j.$$ Since $R_j\rightarrow\infty \mbox{ as } j\rightarrow\infty,$ $g_j(\zeta)$ is defined on whole of $\mathbb{C}.$ Now, let $R>0$ be fixed. Then for $|\zeta|\leq R,$ we have
\begin{align*}
E_{M}(g_j(\zeta), dg_j(\zeta)(1)) &= \rho_j\varphi(\|z_j\|)E_{M}(f(z_j+\rho_j\varphi(\|z_j\|)\xi_j\zeta), df(z_j+\rho_j\varphi(\|z_j\|)\xi_j\zeta))\\
&= \frac{E_{M}(f(z_j+\rho_j\varphi(\|z_j\|)\xi_j\zeta), df(z_j+\rho_j\varphi(\|z_j\|)\xi_j\zeta))}{E_M(f(z_j), df(z_j)(\xi_j))}\\
&\leq \frac{t_j\varphi(\|\tilde{z_j}\|)-\|z_j-\tilde{z_j}\|}{t_j\varphi(\|\tilde{z_j}\|)-\|z_j+\rho_j\varphi(\|z_j\|)\xi_j\zeta-\tilde{z_j}\|}\\
&\leq \frac{t_j\varphi(\|\tilde{z_j}\|)-\|z_j-\tilde{z_j}\|}{t_j\varphi(\|\tilde{z_j}\|)-\|z_j-\tilde{z_j}\|-\rho_j\varphi(\|z_j\|)R}\\
&= \frac{R_j\rho_j\varphi(\|z_j\|)}{R_j\rho_j\varphi(\|z_j\|)-\rho_j\varphi(\|z_j\|)R}\\
&= \frac{1}{1-R/R_j}\\
&< 2, \
\end{align*}
for sufficiently large $j.$ Thus, by Lemma \ref{lem:1}, the family $\{g_j(\zeta)\}$ is normal on $\mathbb{C}.$ Passing to a subsequence, if necessary, we can assume that $\{g_j(\zeta)\}$ converges locally uniformly on $\mathbb{C}$ to a holomorphic mapping $g: \mathbb{C}\rightarrow M.$ Clearly, $g$ is non-constant since $E_{M}(g_j(0), dg_j(0)(1))= \rho_j\varphi(\|z_j\|)E_{M}(f(z_j), df(z_j)(\xi_j))=1.$\\
Conversely, suppose that there exist sequences $\left\{z_j\right\},~\left\{\xi_j\right\}\subset\mathbb{C}^{m},~\left\{\rho_j\right\}\subset(0,1)$ with $\|z_j\|\rightarrow\infty,~\rho_j\rightarrow 0$ and $\|\xi_j\|=1$ such that the sequence $$g_j(\zeta):=f(z_j+\rho_j\varphi(\|z_j\|)\xi_j\zeta)$$ converges uniformly on compact subsets of $\mathbb{C}$ to a non-constant holomorphic mapping $g:\mathbb{C}\rightarrow M.$ We need to show that $f\not\in \mathcal{Y}_{\varphi}(\mathbb{C}^{m}, M).$ Suppose, on the contrary, that $f\in \mathcal{Y}_{\varphi}(\mathbb{C}^{m}, M).$ Then there exists a constant $K_0>0$ such that
\begin{equation}\label{eq:3}
\varphi(\|z\|)E_{M}(f(z), df(z)(\xi))\leq K_0
\end{equation}
 for $z,~\xi\in\mathbb{C}^{m}$ with $\|\xi_j\|=1.$\\
Fix $R>0$ and let $w_j:= z_j+\rho_j\varphi(\|z_j\|)\xi_j\zeta$ for $|\zeta|<R.$ Then from \eqref{eq:3}, we have
\begin{align*}
E_M(g_j(\zeta), dg_j(\zeta)(1)) &= \rho_j\varphi(\|z_j\|)E_M(f(w_j), df(w_j)(\xi_j))\\
&\leq\rho_j\cdot\frac{\varphi(\|z_j\|)}{\varphi(\|w_j\|)}\cdot K_0\\
&\rightarrow 0 \mbox{ as } j\rightarrow\infty,
\end{align*}
since $\varphi(\|w_j\|)/\varphi(\|z_j\|)\rightarrow 1$ as $j\rightarrow\infty.$\\
Thus, $E_M(g_j(\zeta), dg_j(\zeta)(1))=0$ for $|\zeta|<R.$ However, $R>0$ was arbitrary, so we must have $E_M(g_j(\zeta), dg_j(\zeta)(1))=0,~\forall~\zeta\in\mathbb{C},$ which contradicts the fact that $g$ is non-constant. This completes the proof.
\end{proof}

\section{\bf A five point theorem for weighted Yosida mappings}
In this section, we obtain some criteria of holomorphic mappings belonging to the class $\mathcal{Y}_{\varphi}(\mathbb{C}^{m}, P^{1}(\mathbb{C}))$ corresponding to the five point theorem of Lappan  \cite{lappan}. The Lappan's five point theorem states that a family of meromorphic functions on a domain $D\subset\mathbb{C}$ is normal on $D$ if and only there for each compact set $K\subset D,$ there exist a constant $M=M(K)>0$ and a set $E=E(K)\subset\mathbb{C}_{\infty}$ consisting of at least five distinct points such that $$f^{\#}(z)<M, ~z\in K\cap f^{-1}(E)$$ holds for every $f\in\mathcal{F}.$

In recent times, many higher dimensional analogues of the Lappan's result have been obtained (see \cite{charak, chen, datt, hahn, hu, tan}). Following the ideas of Lappan \cite{lappan}, Makhmutov \cite{makhmutov-2} proved a five point theorem for $p$-Yosida functions.
Recently, Yang \cite{yang} obtained a version of Lappan's five point theorem for $p$-Yosida mappings in several complex variables. The following result is an several complex variable analogue of the five point theorem for weighted Yosida mappings. Here, we assume that the reader is familiar with the standard notations and results of the Nevanlinna's value distribution theory of meromorphic functions. For details, we refer the reader to \cite{hayman}.

\begin{thm}\label{thm:2}
Let $f\in\mathcal{H}(\mathbb{C}^{m}, P^{1}(\mathbb{C}))$ and suppose that there exists a set $A\subset P^{1}(\mathbb{C})$ consisting of five distinct points such that $$\sup\limits_{z\in f^{-1}(A),~\|\xi\|=1}\varphi(\|z\|)E_{P^{1}(\mathbb{C})}(f(z), df(z)(\xi))<\infty.$$ Then $f\in\mathcal{Y}_{\varphi}(\mathbb{C}^{m}, P^{1}(\mathbb{C})).$
\end{thm}

\begin{proof}
Assume that $f\notin\mathcal{Y}_{\varphi}(\mathbb{C}^{m}, P^{1}(\mathbb{C})).$ Then by Theorem \ref{thm:1}, there exist sequences $\{z_j\},~\{\xi_j\}\subset\mathbb{C}^{m}$ with $\|z_j\|\rightarrow\infty$ and $\|\xi\|=1,$ positive numbers $\rho_j\rightarrow 0$ such that the sequence $$g_j(\zeta):=f(z_j+\rho_j\varphi(\|z_j\|)\xi_j\zeta)$$ converges uniformly on compact subsets of $\mathbb{C}$ to a non-constant mapping $g\in\mathcal{H}(\mathbb{C}, P^{1}(\mathbb{C})).$\\
 Let $A:=\{a_1, a_2, a_3, a_4, a_5\}\subset P^{1}(\mathbb{C}).$\\
{\bf Claim:} For any $a\in A,$ all zeros of $g-a$ have multiplicity at least $2.$\\
For, let $\zeta_0\in\mathbb{C}$ be such that $g(\zeta_0)-a=0.$ Since $g\not\equiv a,$ by Hurwitz's theorem, there exists $\zeta_j\rightarrow\zeta_0$ such that $g_j(\zeta_j)-a=0.$ That is, $f(z_j+\rho_j\varphi(\|z_j\|)\xi_j\zeta_j)=a.$ By hypothesis, there exists a constant $K>0$ such that $$\varphi(\|w_j\|)E_{P^{1}(\mathbb{C})}(f(w_j), df(w_j)(\xi_j))<K,$$  where $w_j:=z_j+\rho_j\varphi(\|z_j\|)\xi_j\zeta_j.$ Thus, 
\begin{align*}
E_{P^{1}(\mathbb{C})}(g_j(\zeta_j), dg_j(\zeta_j)(1)) &=\rho_j\varphi(\|z_j\|)E_{P^{1}(\mathbb{C})}(f(w_j), df(w_j)(\xi_j))\\
&\leq \rho_j\cdot\frac{\varphi(\|z_j\|)}{\varphi(\|w_j\|)}\cdot K\\
&\rightarrow 0 \mbox{ as } j\rightarrow\infty.
\end{align*}
Hence, $E_{P^{1}(\mathbb{C})}(g_j(\zeta_0), dg_j(\zeta_0)(1))=\lim\limits_{j\to\infty}E_{P^{1}(\mathbb{C})}(g_j(\zeta_j), dg_j(\zeta_j)(1))=0,$ showing that $g'(\zeta_0)=0.$ Thus, $\zeta_0$ is a zeros of $g-a$ with multiplicity at least $2.$ This establishes the claim. Now, by the second fundamental theorem of Nevanlinna, we get
\begin{align*}
3T(r, g) &\leq \sum\limits_{i=1}^{5}\overline{N}\left(r, \frac{1}{g-a_i}\right) + S(r, g)\\
&\leq \frac{1}{2}\sum\limits_{i=1}^{5}N\left(r, \frac{1}{g-a_i}\right) + S(r, g)\\
&\leq \frac{5}{2}T(r, g) + S(r, g),
\end{align*}
a contradiction to the fact that $g$ is non-constant. Hence $f\in\mathcal{Y}_{\varphi}(\mathbb{C}^{m}, P^{1}(\mathbb{C})).$
\end{proof}

It is noteworthy to mention that the cardinality of the set $A$ in Theorem \ref{thm:2} can be reduced under suitable conditions. Tan and the present second author \cite{tan-thin} proved an analogue of Lappan's theorem for normal functions by taking a four point set in one complex variable. Subsequently, several versions of the Lappan's theorem for a four point set were proved by many authors in one as well as in higher dimensions (see \cite{charak, datt, hu, tan, yang}). Motivated by these works, we reduce the cardinality of the set $A$ in Theorem \ref{thm:2} and obtain the following result: 

\begin{thm}\label{thm:3}
Let $f\in\mathcal{H}(\mathbb{C}^{m}, P^{1}(\mathbb{C}))$ and suppose that there exists a set $A\subset P^{1}(\mathbb{C})$ consisting of four distinct points such that 
\begin{equation}
\sup\limits_{z\in f^{-1}(A),~\|\xi\|=1}\varphi(\|z\|)E_{P^{1}(\mathbb{C})}(f(z), df(z)(\xi))<\infty
\end{equation}
 and 
\begin{equation}\label{eq:4}
\sup\limits_{z\in f^{-1}(A\setminus\{\infty\}),~\|\xi\|=1}(\varphi(\|z\|))^{2}E_{P^{1}(\mathbb{C})}\left(\sum\limits_{i=1}^{m}f_{z_i}(z), d\sum\limits_{i=1}^{m}f_{z_i}(z)(\xi)\right)<\infty,
\end{equation}
 where $f_{z_i}=\partial{f}/\partial{z_i},~i=1,2,\ldots, m.$ Then $f\in\mathcal{Y}_{\varphi}(\mathbb{C}^{m}, P^{1}(\mathbb{C})).$
\end{thm}

\begin{proof}
Suppose that $f\notin\mathcal{Y}_{\varphi}(\mathbb{C}^{m}, P^{1}(\mathbb{C})).$ Then by Theorem \ref{thm:1}, there exist sequences $\{z_j\},~\{\xi_j\}\subset\mathbb{C}^{m}$ with $\|z_j\|\rightarrow\infty$ and $\|\xi_j\|=1,$ positive numbers $\rho_j\rightarrow 0$ such that the sequence $$g_j(\zeta):=f(z_j+\rho_j\varphi(\|z_j\|)\xi_j\zeta)$$ converges uniformly on compact subsets of $\mathbb{C}$ to a non-constant mapping $g\in\mathcal{H}(\mathbb{C}, P^{1}(\mathbb{C})).$ Then by applying the same arguments as in the proof of Theorem \ref{thm:2}, one can easily see that for any $a\in S,$ all zeros of $g-a$ have multiplicity at least $2.$ Next, we claim that for any $a\in S\setminus\{\infty\},$ all zeros of $g-a$ have multiplicity at least $3.$ So, let $\zeta_0'\in\mathbb{C}$ such that $g(\zeta_0')-a=0.$ Since $g\not\equiv a,$ by Hurwitz's theorem, there exists $\zeta_j'\rightarrow\zeta_0'$ such that $f(w_j')=a,$ where $w_j':=z_j+\rho_j\varphi(\|z_j\|)\xi_j\zeta_j'.$ Now, by \eqref{eq:4}, there exists a constant $K'>0$ such that 
\begin{equation}\label{eq:5}
(\varphi(\|w_j'\|))^{2}E_{P^{1}(\mathbb{C})}\left(\sum\limits_{i=1}^{m}f_{z_i}(w_j'), d\sum\limits_{i=1}^{m}f_{z_i}(w_j')(\xi_j)\right)\leq K'.
\end{equation}
Then for $j\in\mathbb{N}$ and $\xi_j=(\xi_{j1},\ldots, \xi_{jm}),$ we have $$g'_{j}(\zeta)=\rho_j\varphi(\|z_j\|)\sum\limits_{i=1}^{m}\xi_{ji}f_{z_i}(z_j+\rho_j\varphi(\|z_j\|)\xi_j\zeta).$$

A simple computation together with \eqref{eq:5} shows that
\begin{align*}
E_{P^{1}(\mathbb{C})}(g'{_j}(\zeta_{j}'), dg'_{j}(\zeta_{j}')(1)) &=\rho_{j}^{2}(\varphi(\|z_j\|)^{2}\sum\limits_{i=1}^{m}|\xi_{ji}|E_{P^{1}(\mathbb{C})}(f_{z_i}(w_j'), df_{z_i}(w_j')(\xi_j))\\
&\leq \rho_{j}^{2}(\varphi(\|z_j\|)^{2}E_{P^{1}(\mathbb{C})}\left(\sum\limits_{i=1}^{m}f_{z_i}(w_j'), d\sum\limits_{i=1}^{m}f_{z_i}(w_j')(\xi_j)\right)\\
&\leq \rho_{j}^{2}\cdot\frac{(\varphi(\|z_j\|))^{2}}{(\varphi(\|w_j'\|))^{2}}\cdot K'\\
&\rightarrow 0 \mbox{ as } j\rightarrow\infty.
\end{align*}
Hence, $E_{P^{1}(\mathbb{C})}(g'_j(\zeta_0'), dg'_j(\zeta_0')(1))=\lim\limits_{j\to\infty}E_{P^{1}(\mathbb{C})}(g'_j(\zeta_j'), dg'_j(\zeta_j')(1))=0,$ showing that $g''(\zeta_0')=0.$ Thus, for any $a\in A\setminus\{\infty\},$ all zeros of $g-a$ have multiplicity at least $3.$ Let $A=\{a_1, a_2, a_3, a_4\}\subset P^{1}(\mathbb{C}),$ where $a_4$ is either finite or infinite. Appealing to the second fundamental theorem of Nevanlinna, we obtain 
\begin{align*}
2T(r, g) &\leq \sum\limits_{i=1}^{4}\overline{N}\left(r, \frac{1}{g-a_i}\right) + S(r, g)\\
&\leq \frac{1}{3}\sum\limits_{i=1}^{3}N\left(r, \frac{1}{g-a_i}\right) + \frac{1}{2}N\left(r, \frac{1}{g-a_4}\right) + S(r, g)\\
&\leq \frac{3}{2}T(r, g) + S(r, g),
\end{align*}
a contradiction to the fact that $g$ is non-constant. Hence $f\in\mathcal{Y}_{\varphi}(\mathbb{C}^{m}, P^{1}(\mathbb{C})).$
\end{proof}

\section{\bf Weighted Yosida mappings concerning hypersurfaces in $P^{n}(\mathbb{C})$}
In this section, we prove some criteria of holomorphic mappings intersecting with moving hypersurfaces in $P^{n}(\mathbb{C})$ to be classified as weighted Yosida mappings. Before stating our results, we pause to give a some basic notions and preliminary results.

\medskip

Let $D$ be a nonempty open connected set in $\mathbb{C}^{m}$ and let $f\in\mathcal{H}(D, P^{n}(\mathbb{C})).$  A holomorphic mapping $\hat{f}=(f_0,\ldots, f_n): D\rightarrow\mathbb{C}^{n+1}$ is said to be a representation of $f$ on $D$ if $\pi(\hat{f}(z))=f(z)$ for all $z\in D,$ where $\pi:\mathbb{C}^{n+1}\setminus\{0\}\rightarrow P^{n}(\mathbb{C})$ is the standard quotient mapping. Further, if the set $I(f):=\{z\in D: f_0(z)=\cdots=f_n(z)=0\}$ is empty, then we say that $\hat{f}$ is a reduced representation of $f$ on $D.$ It is worthwhile to mention that for any $a\in D,$ $f$ always has a reduced representation in a neighbourhood of $a.$

\smallskip

Denote by $H_{D},$ the ring of all holomorphic functions on $D$ and denote by $H_D[x_0,\ldots, x_n],$ the set of all homogeneous polynomials in variable $x=(x_0,\ldots, x_n)\in\mathbb{C}^{n+1}$ over $H_D.$ Let $Q\in H_D[x_0,\ldots, x_n]\setminus\{0\}$ be a homogeneous polynomials of degree $d$ and assume that the coefficients of $Q$ have no common zeros. Then we can write $$Q(x)=Q(x_0,\ldots, x_n)=\sum\limits_{j=0}^{n_d}a_jx^{I_j}=\sum\limits_{j=0}^{n_d}a_jx_{0}^{i_{j0}}\cdots x_n^{i_{jn}},$$
where $a_j\in H_D,~I_j=(i_{j0},\ldots, i_{jn})$ with $|I_j|=i_{j0}+\cdots+i_{jn}=d$ for $j=0,\ldots, n_d$ and $n_d=\binom{n+d}{n}-1.$ Thus, for each $z\in D,$ there is a homogeneous polynomials $Q(z)$ over $\mathbb{C}^{n+1}$ defined by $$Q(z)(x)=\sum\limits_{j=0}^{n_d}a_j(z)x^{I_j}$$ and hence, to each $z\in D,$ there is an associated hypersurface $S$ given by $$S(z)=\{x\in\mathbb{C}^{n+1}: Q(z)(x)=0\}.$$ We call this hypersurface a {\it moving hypersurface} in $P^{n}(\mathbb{C})$ defined by the homogeneous polynomial $Q.$ Since the coefficients of $Q$ have no common zeros, the hypersurface $S(z)$ is well defined for each $z\in D.$ 

\begin{definition}(see \cite{chen, dethloff})
Let $S_j,~j=1,\ldots, q$ be $q~(\geq n+1)$ moving hypersurfaces in $P^{n}(\mathbb{C})$ defined by the homogeneous polynomials $Q_j\in H_D[x_0,\ldots, x_n].$ Then we say that these moving hypersurfaces $S_j$ are located in (weakly) general position if there exists $z\in D$ such that, for any $1\leq j_0<\cdots <j_n\leq q,$ the system of equations 
$$\left\{\begin{array}{cc} Q_{j_i}(z)(x_0,\ldots, x_n)=0,\\ 0\leq i\leq n\end{array}\right.$$ has only the trivial solution $x=(0,\ldots, 0)\in\mathbb{C}^{n+1}.$ This is equivalent to $$S(Q_1,\ldots, Q_q)(z)=\prod\limits_{1\leq j_0<\cdots <j_n\leq q}\inf\limits_{\|x\|=1}(|Q_{j_0}(z)(x)|^2+\cdots+|Q_{j_n}(z)(x)|^2)>0,$$ where $x=(x_0,\ldots, x_n),~Q_j(z)(x)=\sum\limits_{|I|=d_j}a_{j_I}x^{I}$and $\|x\|=(\sum\limits_{j=0}^{n}|x_j|^2)^{1/2}.$
\end{definition}

\medskip

Let $\mathcal{S}=\{S_j\}_{j=1}^{q}$ be a collection of moving hypersurfaces defined by the homogeneous polynomials $Q_j\in H_D[x_0,\ldots, x_n]$ with $\deg{Q_j}=d_j.$ Let $l_{\mathcal{S}}=\lcm(d_1,\ldots, d_q)$ and set $n_{\mathcal{S}}=\binom{n+l_{\mathcal{S}}}{n}-1.$

Recall that for $I=(i_0,\ldots, i_n),~J=(j_0,\ldots, j_n)\in\mathbb{N}^{n+1},$ we say that $I<J$ if and only if there exists $k\in\{0,\ldots, n\}$ such that $i_{t}=j_t$ for $t<k$ and $i_k<j_k.$ Now, for $l=0,\ldots, n_{\mathcal{S}},$ let $$\left\{\begin{array}{ccc}I_0 =\{i_{0_0},\ldots, i_{0_n}\},\\
I_1 =\{i_{1_0},\ldots, i_{1_n}\},\\
 \vdots\\
I_{n_{\mathcal{S}}} =\{i_{n_{\mathcal{S}_{0}}},\ldots, i_{n_{\mathcal{S}_{n}}}\}\end{array}\right\}\subset\mathbb{N}^{n+1}$$ with $|I_l|=\sum_{k=0}^{n}i_{l_{k}}=l_{\mathcal{S}}$ and $I_s<I_t$ for $s<t.$\\
Let $[x_0:\cdots: x_n]$ and $[w_0:\cdots: w_{n_{\mathcal{S}}}]$ be the homogeneous coordinates in $P^{n}(\mathbb{C})$ and $P^{n_{\mathcal{S}}}(\mathbb{C})$ respectively. Then the map $\nu_{l_{\mathcal{S}}}: P^{n}(\mathbb{C})\rightarrow P^{n_{\mathcal{S}}}(\mathbb{C})$ given by $$\nu_{l_{\mathcal{S}}}(x)=(w_0(x):\cdots: w_{n_{\mathcal{S}}}(x)), \mbox{ where } w_l(x):=x^{I_l}=x_0^{i_{l_{0}}}x_1^{i_{l_{1}}}\cdots x_n^{i_{l_{n}}},~ l=0,\ldots, n_{\mathcal{S}},$$ is called the {\it Veronese embedding} of degree $l_{\mathcal{S}}.$ 

\smallskip

For any $S_j\in\mathcal{S},$ we set $Q_j^*=Q_j^{l_{\mathcal{S}}/d_j}$ and let $b_j=(b_{j_0},\ldots, b_{j_{n_{\mathcal{S}}}})$ be the vector associated with $Q_j^*.$ Then $$L_j=Q_j^*=b_{j_0}w_0+\cdots b_{j_{n_{\mathcal{S}}}}w_{n_{\mathcal{S}}}$$ is a moving linear form in $P^{n_{\mathcal{S}}}(\mathbb{C}).$ \\
Let $H_j$ be a hyperplane in $P^{n_{\mathcal{S}}}(\mathbb{C})$ defined by the linear form $L_j.$ Then for the collection $\mathcal{S}=\{S_j\}_{j=1}^{q}$ of moving hypersurfaces in $P^{n}(\mathbb{C}),$ there is a corresponding collection $\mathcal{H}=\{H_j\}_{j=1}^{q}$ of associated moving hyperplanes in $P^{n_{\mathcal{S}}}(\mathbb{C}).$

\begin{definition}
let $q\geq n_{\mathcal{S}}+1.$ Then the collection $\mathcal{S}=\{S_j\}_{j=1}^{q}$ of moving hypersurfaces is said to be located in (weakly) general position for Veronese embedding in $P^{n}(\mathbb{C})$ if there exists $z\in D$ such that the corresponding collection $\mathcal{H}=\{H_j(z)\}_{j=1}^{q}$ of associated moving hyperplanes are in general position in $P^{n_{\mathcal{S}}}(\mathbb{C})$ at $z.$
\end{definition} 

In addition to the above notions, we need some results from the Nevanlinna theory in several complex variables (see \cite{noguchi}). The following result established by Nochka \cite{nochka} is a Picard type theorem for holomorphic curves: 

\begin{lem}\cite{nochka}\label{lem:nochka}
Let $f\in\mathcal{H}(\mathbb{C}, P^{n}(\mathbb{C})),$ $H_1,\ldots, H_q,~q\geq 2n+1$ be $q$ hyperplanes in $P^{n}(\mathbb{C})$ in general position and let $l_i\in\mathbb{N}\cup\{\infty\}~i=1,\ldots,q$ such that $$\sum\limits_{i=1}^{q}\frac{1}{l_i}<\frac{q-n-1}{n}.$$ If $f$ intersects $H_i$ with multiplicity at least $l_i,$ then $f$ reduces to a constant. 
\end{lem}

We also need the following Picard type theorem for moving hypersurfaces in $P^{n}(\mathbb{C}):$
 
\begin{lem}\cite{hu}\label{lem:hu-thin}
Let $f:\mathbb{C}^{m}\rightarrow P^{n}(\mathbb{C})$ be a meromorphic mapping with reduced representation $\tilde{f}.$ Let $q\in\mathbb{N},~l_i\in\mathbb{N}\cup\{\infty\},~i=1,\ldots, q$ and suppose that there exist $q$ homogeneous polynomials $Q_i\in\mathbb{C}[x_0,\ldots, x_n]$ with $\deg Q_i=d_i,~i=1,\ldots, q$ and $Q_i(\tilde{f})\not\equiv 0$ such that the hypersurfaces $S_i$ associated to $Q_i$ are in weakly general position. Assume that $q\geq nn_{\mathcal{S}}+n+1$ and $$\sum\limits_{i=1}^{q}\frac{1}{l_i}<\frac{q-(n-1)(n_{\mathcal{S}}+1)}{n_S(n_{\mathcal{S}}+2)}, \mbox{ where } n_{\mathcal{S}}=\binom{n+l_{\mathcal{S}}}{n}-1 \mbox{ and } l_{\mathcal{S}}=\lcm(d_1,\ldots,d_q).$$ If $f$ intersects $S_i$ with multiplicity at least $l_i,$ then $f$ must be a constant.
\end{lem}

Now, we are in a position to formulate our results.

\begin{thm}\label{thm:4}
Let $f\in\mathcal{H}(\mathbb{C}^{m}, P^{n}(\mathbb{C}))$ and let $S_1,\ldots, S_q$ be $q~(\geq 2n_{\mathcal{S}}+1)$ hypersurfaces, which are located in general position for Veronese embedding in $P^{n}(\mathbb{C}),$ defined by homogeneous polynomials $Q_j\in\mathbb{C}[x_0,\ldots,x_n],~j=1,\ldots, q$ having degree $d_j.$ Let $\hat{f}=(f_0,\ldots,f_n)$ be a reduced representation of $f$ on $\mathbb{C}^{m}$ with $f_i(z)\neq 0$ for some $i$ and $z,$ such that for all $j=1,\ldots, q$ and $l_j\geq 2,$ $$\sup\limits_{1\leq |\alpha|\leq l_j-1,~z\in f^{-1}(S_j)}\varphi^{|\alpha|}(\|z\|)\left|\frac{\partial^{\alpha}(Q_j(\hat{f_i}))}{\partial z^{\alpha}}(z)\right|<\infty,$$ where $l_j\in\mathbb{N}\cup\{\infty\}$ with $$\sum\limits_{j=1}^{q}\frac{d_j}{l_j}<\frac{(q-n_{\mathcal{S}}-1)l_{\mathcal{S}}}{n_{\mathcal{S}}},~ n_{\mathcal{S}}=\binom{n+l_{\mathcal{S}}}{n}-1 \mbox{ and } l_{\mathcal{S}}=\lcm(d_1,\ldots,d_q).$$  Then $f\in\mathcal{Y}_{\varphi}(\mathbb{C}^{m}, P^{n}(\mathbb{C})).$
\end{thm}

\begin{proof}
Suppose that $f\not\in\mathcal{Y}_{\varphi}(\mathbb{C}^{m}, P^{n}(\mathbb{C})).$ Then by Theorem \ref{thm:1}, there exist sequences $\{z_t\},~\{\xi_t\}\subset\mathbb{C}^{m}$ with $\|z_t\|\rightarrow\infty$ and $\|\xi_t\|=1,$ $\{\rho_t\}\subset (0,1)$ with $\rho_t\rightarrow 0$ such that $$g_t(\zeta):=f(z_t+\rho_t\varphi(\|z_t\|)\xi_t\zeta)$$ converges uniformly on compact subsets of $\mathbb{C}$ to a non-constant mapping $g\in\mathcal{H}(\mathbb{C}, P^{n}(\mathbb{C})).$ Therefore, there exists a reduced representation $\hat{g}=(g_0,\ldots, g_n):\mathbb{C}\rightarrow\mathbb{C}^{n+1}$ such that for each $j\in\{1,\ldots,q\}$ and as $t\rightarrow\infty,$ we have $$\hat{g}_t(\zeta)=(g_{0t}(\zeta),\ldots,g_{nt}(\zeta))\rightarrow\hat{g}(\zeta) \mbox{ and } Q_j(\hat{g}_t(\zeta))\rightarrow Q_j(\hat{g}(\zeta))$$ uniformly on compact subsets of $\mathbb{C},$ where $\hat{g}_t(\zeta)=\hat{f}(z_t+\rho_t\varphi(\|z_t\|)\xi_t\zeta).$\\
{\bf Claim:} For each $j=1,\ldots,q,$ $g(\mathbb{C})$ intersects $S_j$ with multiplicity at least $l_j.$\\
Let $\zeta_0\in\mathbb{C}$ such that $Q_j(\hat{g}(\zeta_0))=0.$ Then there exist a disk $D(\zeta_0, r_0)\subset\mathbb{C}$ and an index $i\in\{0,\ldots,n\}$ such that $g(D(\zeta_0, r_0))\subset\{[x_0,\ldots,x_n]\in P^{n}(\mathbb{C}): x_i\neq 0\}$ and $g_i(\zeta)\neq 0$ on $D(\zeta_0, r_o).$ Thus by Hurwitz's theorem, for sufficiently large $t,$ $g_{it}(\zeta)\neq 0$ on $D(\zeta_0, r_0)$ and hence $\hat{g_{it}}(\zeta):=\left(\frac{g_{0t}}{g_{it}},\ldots,\frac{g_{nt}}{g_{it}}\right)\rightarrow \hat{g_i}(\zeta)$ uniformly on $D(\zeta_0, r_0).$ This further implies that $Q_j(\hat{g_{it}}(\zeta))\rightarrow Q_j(\hat{g_{i}}(\zeta))$ uniformly on $D(\zeta_0, r_0)$ and so for each $k\in\mathbb{N},$ $(Q_j(\hat{g_{it}}(\zeta)))^{(k)}\rightarrow (Q_j(\hat{g_{i}}(\zeta)))^{(k)}$ uniformly on $D(\zeta_0, r_0).$
By hypothesis, for all $j=1,\ldots, q$ and when $l_j\geq 2,$ we have $$\sup\limits_{1\leq|\alpha|\leq l_j-1,~z\in f^{-1}(S_j)}\varphi^{|\alpha|}(\|z\|)\left|\frac{\partial^{\alpha}(Q_j(\hat{f_{i}}))}{\partial z^{\alpha}}(z)\right|<\infty,$$ where $f_{i}(z)\neq 0$ and $S_{j}$ is the hypersurface associated to $Q_j.$ Since $Q_j(\hat{g}(\zeta_0))=0,$ by Hurwitz's theorem, there exists $\zeta_t\rightarrow\zeta_0$ such that for sufficiently large $t,$ $Q_j(\hat{g_{t}}(\zeta_t))=0.$ That is, $Q_j(\hat{f}(z_t+\rho_t\varphi(\|z_t\|)\xi_t\zeta_t))=0$ and so $z_t+\rho_t\varphi(\|z_t\|)\xi_t\zeta_t)\in f^{-1}(S_j).$ Thus there exists $K>0$ such that for sufficiently large $t,$ and when $l_j\geq 2,$ $$\varphi^{|\alpha|}(\|z_t+\rho_t\varphi(\|z_t\|)\xi_t\zeta_t)\|)\left|\frac{\partial^{\alpha}(Q_j(\hat{f_{i}}))}{\partial z^{\alpha}}(z_t+\rho_t\varphi(\|z_t\|)\xi_t\zeta_t))\right|<K$$ for some $i\in\{0,\ldots, n\}$ and all $\alpha:=(\alpha_1,\ldots,\alpha_m)$ with $1\leq |\alpha|\leq l_j-1.$ Then 
\begin{align*}
\left|(Q_j(\hat{g_{it}}(\zeta_t)))^{(|\alpha|)}\right|&=\left|\sum\limits_{\alpha}C_\alpha(\rho_t\varphi(\|z_t\|))^{|\alpha|}\frac{\partial^{\alpha}(Q_j(\hat{f_{i}}))}{\partial z_{1}^{\alpha_1}\cdots\partial z_{m}^{\alpha_m}}(z_t+\rho_t\varphi(\|z_t\|)\xi_t\zeta_t))\right|\\
&\leq C\rho_{t}^{|\alpha|}K\left(\frac{\varphi(\|z_t\|)}{\varphi(\|z_t+\rho_t\varphi(\|z_t\|)\xi_t\zeta_t)\|)}\right)^{|\alpha|}\\
&\rightarrow 0 \mbox{ as } t\rightarrow\infty,
\end{align*}
where $C,~C_\alpha$ are suitable constants. Thus $$(Q_j(\hat{g_{i}}(\zeta_0)))^{(|\alpha|)}=\lim\limits_{t\to\infty}(Q_j(\hat{g_{it}}(\zeta_t)))^{(|\alpha|)}=0$$ for $1\leq |\alpha|\leq l_j-1,$ showing that $\zeta_0$ is a zero of $Q_j(\hat{g})$ with multiplicity at least $l_j$ and this establishes the claim. Let $\nu_{l_{\mathcal{S}}}: P^{n}(\mathbb{C})\rightarrow P^{n_{\mathcal{S}}}(\mathbb{C})$ be the Veronese embedding of degree $l_{\mathcal{S}}.$ Then $\psi_t:=\nu_{l_{\mathcal{S}}}(g_t)$ converges uniformly on compact subsets of $D(\zeta_0, r_0)$ to $\psi:=\nu_{l_{\mathcal{S}}}(g_t).$ Let $\hat{\psi}$ be the reduced representation of $\psi.$ Since $Q_j^{l_{\mathcal{S}}/d_j}(\hat{g}(\zeta))=Q_{j}^{*}(\hat{\psi}(\zeta)),$ it follows that for each $j=1,\ldots, q,$ $\psi$ intersects $H_j^{*}$ with multiplicity at least $(l_{\mathcal{S}}/d_j)\cdot l_j,$ where $H_j^{*}$ is the hyperplane associated with the linear form $Q_j^{*}$ in $ P^{n_{\mathcal{S}}}(\mathbb{C}).$ Again, since the hyperplanes $H_1^{*},\ldots, H_{q}^{*}$ are in general position in $ P^{n_{\mathcal{S}}}(\mathbb{C}),$ by hypothesis, we have $$\sum\limits_{j=1}^{q}\frac{d_j}{l_j}<\frac{(q-n_{\mathcal{S}}-1)l_{\mathcal{S}}}{n_{\mathcal{S}}}.$$ Now from Lemma \ref{lem:nochka}, it follows that $\psi$ and hence $g$ must be a constant, which is a contradiction. Thus, $f\in\mathcal{Y}_{\varphi}(\mathbb{C}^{m}, P^{n}(\mathbb{C})).$ 
\end{proof}

\begin{thm}\label{thm:5}
Let $f\in\mathcal{H}(\mathbb{C}^{m}, P^{n}(\mathbb{C}))$ and let $S_1,\ldots, S_q$ be $q~(\geq nn_{\mathcal{S}}+2n+1)$ hypersurfaces, which are located in general position in $P^{n}(\mathbb{C}),$ defined by homogeneous polynomials $Q_j\in\mathbb{C}[x_0,\ldots,x_n],~j=1,\ldots, q$ having degree $d_j.$ Let $\hat{f}=(f_0,\ldots,f_n)$ be a reduced representation of $f$ on $\mathbb{C}^{m}$ with $f_i(z)\neq 0$ for some $i$ and $z,$ such that for all $j=1,\ldots, q$ and $l_j\geq 2,$ $$\sup\limits_{1\leq |\alpha|\leq l_j-1,~z\in f^{-1}(S_j)}\varphi^{|\alpha|}(\|z\|)\left|\frac{\partial^{\alpha}(Q_j(\hat{f_i}))}{\partial z^{\alpha}}(z)\right|<\infty,$$ where $l_j\in\mathbb{N}\cup\{\infty\}$ with $$\sum\limits_{j=1}^{q}\frac{1}{l_j}<\frac{(q-n-(n-1)(n_{\mathcal{S}}+1)}{n_{\mathcal{S}}(n_{\mathcal{S}}+2)},~ n_{\mathcal{S}}=\binom{n+l_{\mathcal{S}}}{n}-1 \mbox{ and } l_{\mathcal{S}}=\lcm(d_1,\ldots,d_q).$$  Then $f\in\mathcal{Y}_{\varphi}(\mathbb{C}^{m}, P^{n}(\mathbb{C})).$
\end{thm}

\begin{proof}
Suppose, on the contrary, that $f\not\in\mathcal{Y}_{\varphi}(\mathbb{C}^{m}, P^{n}(\mathbb{C})).$ Then by Theorem \ref{thm:1}, there exist sequences $\{z_t\},~\{\xi_t\}\subset\mathbb{C}^{m}$ with $\|z_t\|\rightarrow\infty$ and $\|\xi_t\|=1,$ $\{\rho_t\}\subset (0,1)$ with $\rho_t\rightarrow 0$ such that $$g_t(\zeta):=f(z_t+\rho_t\varphi(\|z_t\|)\xi_t\zeta)$$ converges uniformly on compact subsets of $\mathbb{C}$ to a non-constant mapping $g\in\mathcal{H}(\mathbb{C}, P^{n}(\mathbb{C})).$ Let $\hat{g}$ be a reduced representation of $g$ and let $J=\{j\in\{1,\ldots, q\}: Q_j(\hat{g})\equiv 0\}.$ Then $0\leq |J|\leq n,$ owing to the fact that the hypersurfaces $S_1,\ldots, S_q$ are in general position in $P^{n}(\mathbb{C})$ and so at most $n$ of these hypersurfaces can contain the image of $g.$ Since $q\geq nn_{\mathcal{S}}+2n+1,$ it follows that $q-|J|\geq nn_{\mathcal{S}}+n+1.$ By applying the same arguments as in the proof of Theorem \ref{thm:4}, one can easily see that $g(\mathbb{C})$ intersects with the hypersurfaces $\{S_j\}_{j\not\in J}$ with multiplicity at least $l_j.$ Further, a simple computation shows that 
\begin{align*}
\sum\limits_{j\not\in J}\frac{1}{l_j}\leq\sum\limits_{j=1}^{q}\frac{1}{l_j} &< \frac{(q-n-(n-1)(n_{\mathcal{S}}+1)}{n_{\mathcal{S}}(n_{\mathcal{S}}+2)}\\
&\leq \frac{(q-|J|-(n-1)(n_{\mathcal{S}}+1)}{n_{\mathcal{S}}(n_{\mathcal{S}}+2)}.
\end{align*}
This along with Lemma \ref{lem:hu-thin} imply that $g$ must be a constant, which is not in reason. Hence $f\in\mathcal{Y}_{\varphi}(\mathbb{C}^{m}, P^{n}(\mathbb{C})).$
\end{proof}

\bibliographystyle{amsplain}

\end{document}